\documentclass[10pt,reqno]{amsart}
\usepackage{enumitem}
\usepackage[pdfusetitle,colorlinks]{hyperref}
\usepackage[british]{babel}

\usepackage{amssymb}
\usepackage{amsthm}
\usepackage{mathtools}

\usepackage[capitalise]{cleveref}
\usepackage[backend=bibtex,isbn=false,maxnames=7,sortcites]{biblatex}
\DeclareSourcemap{
  \maps[datatype=bibtex]{
    \map{
      \step[fieldsource=doi,final]
      \step[fieldset=url,null]
      \step[fieldset=urldate,null]
    }
  }
}

\bibliography{lit.bib}

\numberwithin{equation}{section}

\newtheorem{theorem}{Theorem}

\newtheorem{lemma}[theorem]{Lemma}

\theoremstyle{definition}

\newtheorem{remark}{Remark}[section]
\newtheorem{hypothesis}{Hypothesis}
\makeatletter
\newcommand{\sethypothesistag}[1]{
  \let\oldthehypothesis\thehypothesis
  \renewcommand{\thehypothesis}{#1}
  \g@addto@macro\endhypothesis{
    \addtocounter{hypothesis}{-1}
    \global\let\thehypothesis\oldthehypothesis}
}
\makeatother

\newcommand{\hypref}[1]{\hyperref[#1]{\normalfont\textbf{(H\ref*{#1})}}}
\renewcommand{\vec}[1]{\mathbf{#1}}

\newcommand{\ee}{\mathrm{e}} 
\newcommand{\dd}{\mathrm{d}} 

\newcommand{\R}{\mathbb{R}} 
\newcommand{\N}{\mathbb{N}} 

\newcommand{\fsL}{\mathsf{L}} 
\newcommand{\fsWL}{{\mathcal L}_\infty} 

\newcommand{\ip}[2]{\left\langle{#1},{#2}\right\rangle}
\newcommand{\lAvg}[1]{\langle{#1}\rangle} 

\newcommand{\diss}{\mathcal D} 

\newcommand{\supp}{\operatorname{supp}} 
\newcommand{\init}{\mathrm{in}} 

\newcommand{\collFP}{L_{\mathrm{FP}}} 
\newcommand{\collBGK}{L_{\mathrm{BGK}}} 

\newcommand{\fSt}{f_{\infty}} 
\newcommand{\sU}{\mathsf U}

\begin{document}

\title{\(\fsL^2\)-Hypocoercivity for non-equilibrium kinetic equations}

\author{Helge Dietert} \address{Université Paris Cité and
  Sorbonne Université, CNRS\\ IMJ-PRG, F-75006 Paris, France.}
\email{helge.dietert@imj-prg.fr}

\begin{abstract}
  The recent work
  \cite{dietert-hérau-hutridurga-mouhot-2022-preprint-quantitative-geometric-control-linear-kinetic-theory}
  developed a general framework to show hypocoercivity for a
  stationary Gibbs state and allowed spatial degeneracy, confining
  potentials and boundary conditions. In this work, we show that the
  explicit energy approach in the weighted \(\fsL^2\) space works for
  general non-equilibrium steady states and that it can be adapted to
  cases with weaker confinement leading to algebraic decay.
\end{abstract}

\maketitle

\section{Introduction}

\subsection{Motivation}

The work
\cite{dietert-hérau-hutridurga-mouhot-2022-preprint-quantitative-geometric-control-linear-kinetic-theory}
studies the hypocoercivity for general linear kinetic equations
allowing spatial degeneracy, confining potentials and boundary
conditions, but focuses on Gibbs states where the transport and
collision operator vanish separately for the stationary state.

The basic idea (see also the exposition
\cite{dietert-hérau-hutridurga-mouhot-2022-preprint-trajectorial}) is
to use the transport to transfer the information in a good region with
dissipation which is ensured by a geometric control condition
generalising the condition found in
\textcite{bernard-salvarani-2013-boltzmann} and
\textcite{han-kwan-leautaud-2015-geometric-boltzmann}.  In the good
region with dissipation, we can reduce the required control to the
spatial density which is then controlled by a construction using the
Bogovskiǐ operator inspired from
\cite{albritton-armstrong-mourrat-novack-2019-preprint-variational-fokker-planck}
and similar to the works
\cite{cao-lu-wang-2023-explicit-l-convergence-rate-estimate,brigati-2021-preprint-time-fokker-planck,brigati-stoltz-2023-preprint-how-fokker}.

In this work, we are interested in the decay towards non-equilibrium
stationary states which are created by a non-isothermal collision
operator. Here the stationary state is not explicit and recent works
on the decay rate have been achieved by Doeblin and Harris theorem,
see e.g.\
\textcite{bernou-2023-preprint-asymptotic-behavior-degenerate-linear-kinetic}
for the non-isothermal problem and \textcite{yoldas-2023-harris} for a
general review on the Harris theorem in hypocoercive equations.  The
energy estimate approach of
\textcite{dolbeault-mouhot-schmeiser-2015-hypocoercivity} has been
extended to some non-explicit states in, e.g.,
\cite{bouin-hoffmann-mouhot-2017-exponential,evans-menegaki-2023-preprint-properties-non-equilibrium-steady-states-bgk}.

The aim of this work is to show that the approach with the Bogovoskiǐ
operator covers these more general cases in a natural way leading to a
robust and quantitative method. Moreover, we will show that it
naturally covers weak confinement forces.  As the treatment of spatial
degeneracy and boundary conditions can be adapted from
\cite{dietert-hérau-hutridurga-mouhot-2022-preprint-quantitative-geometric-control-linear-kinetic-theory},
we focus here on the decay without these difficulties for better
readability.

\subsection{Considered problem}

Assume the evolution of a kinetic density \(f=f(t,x,v)\) for
\(x,v \in \R^d\), \(d \in \N\), given by the linear evolution
\begin{equation}
  \label{eq:kinetic-evolution}
  \partial_t f + v \cdot \nabla_x f + \nabla_v\cdot (G f)
  = Lf,
\end{equation}
where \(G=G(x,v)\) is an arbitrary smooth vector field modelling an
external force and \(L\) is a collision operator. Modelling a
non-isothermal background, we assume a function \(M=M(x,v)\) such that
for every \(x \in \R^d\), the function \(v \mapsto M(x,v)\) is a
probability distribution to which the collision operator \(L\) drives
the density \(f\).

The aim is now to understand the decay behaviour of
\eqref{eq:kinetic-evolution} in view of the degenerate dissipation as
\(L\) only changes the velocity distribution.

For our discussion, we take either a Fokker-Planck like collision
operator \(L = \collFP\) defined by
\begin{equation}
  \label{eq:def-coll-fp}
  \collFP f
  := \nabla_v \left( M \cdot \nabla_v\left(\frac{f}{M}\right) \right)
\end{equation}
or a BGK operator \(L = \collBGK\) defined by
\begin{equation}
  \label{eq:def-coll-bgk}
  \collBGK f
  := \lAvg{f} M - f,
\end{equation}
where we introduced the general definition \(\lAvg{\cdot}\) for the
spatial density, i.e.\
\begin{equation}
  \label{eq:definition-local-density}
  \lAvg{g}(t,x) := \int_{v \in \R^d} g(t,x,v)\, \dd v.
\end{equation}

The first assumption is the existence of a stationary state \(\fSt\).
\begin{hypothesis}[Stationary state]\label{h:stationary}
  Assume a stationary state \(\fSt = \fSt(x,v) \ge 0\) such that
  \begin{equation}
    \label{eq:def-stationary-state}
    v \cdot \nabla_x \fSt + \nabla_v \cdot (G \fSt) = L \fSt
  \end{equation}
  which is normalised as
  \begin{equation}
    \label{eq:normalisation-stationary-state}
    \int_{x,v \in \R^d} \fSt\, \dd x\, \dd v = 1
  \end{equation}
  and satisfies
  \begin{equation}
    \label{eq:lower-bound-stationary-state}
    \inf_{x} \inf_{|v| \le 1} \frac{\fSt(x,v)}{\lAvg{\fSt}(x)}
    =: c_\infty > 0.
  \end{equation}
  If \(L=\collFP\), assume also that
  \begin{equation}
    \label{eq:regularity-m-fp}
    \sup_{x} \sup_{|v| \le 1} \frac{|\nabla_v M|}{M}
    =: c_{\mathrm{MFP}} < \infty
  \end{equation}
  and if \(L=\collBGK\) assume that
  \begin{equation}
    \label{eq:regularity-m-bgk}
    \sup_{x} \int_{|v| \le 1} M^2 \, \frac{\lAvg{\fSt}}{\fSt}\, \dd v
    =: c_{\mathrm{MBGK}} < \infty.
  \end{equation}
\end{hypothesis}

\begin{remark}
  The stationary state can be found by several strategies. A first
  approach is by a perturbative argument of an equilibrium state with
  an explicit solution. Another source of non-equilibrium states in
  kinetic theory is obtained through fixed-point arguments like in
  \textcite{evans-menegaki-2021-existence-bgk}. Yet another approach
  is given through the Krein-Rutman theory, see
  \cite{fonte-gabriel-mischler-2023-preprint-krein-rutman}.
\end{remark}

\begin{remark}
  The conditions \eqref{eq:lower-bound-stationary-state},
  \eqref{eq:regularity-m-fp} and \eqref{eq:regularity-m-bgk} are
  concerned with velocities \(|v| \le 1\). The condition
  \eqref{eq:lower-bound-stationary-state} is used to gain a control on
  the spatial density \(\lAvg f\) and \eqref{eq:regularity-m-fp} and
  \eqref{eq:regularity-m-bgk}, respectively, are used to bound error
  terms.
\end{remark}

The relative entropy is a natural distance towards a stationary state
and in the linearised setting this motivates the use of the weighted
\(\fsL^2\) space \(\fsWL^2(\R^{2d})\) defined by the norm
\begin{equation}
  \label{eq:weighted-l2}
  \| g \|_{\fsWL^2(\R^{2d})}^2 := \int_{x,v \in \R^d} |g(x,v)|^2\,
  \frac{\dd x\, \dd v}{f_\infty}
\end{equation}
for a function \(g : \R^{2d} \to \R\). If we consider a function
\(g=g(t,x,v)\) over a time interval \([0,T]\), we similarly use the
weighted space \(\fsWL^2([0,T]\times \R^{2d})\) defined by
\begin{equation}
  \label{eq:weighted-l2-time}
  \| g \|_{\fsWL^2([0,T]\times \R^{2d})}^2
  := \int_0^T \int_{x,v \in \R^d} |g(t,x,v)|^2\,
  \frac{\dd x\, \dd v}{f_\infty}\dd t.
\end{equation}
Henceforth, we will always work in these weighted spaces.

The proof for exponential decay is based on a local coercivity in the
velocity variable and a spatial coercivity for the density which is
similar to the spatial diffusion in the fluid limit. The obtained
decay rate is then constructive and depends on the coercivity of these
two components.

\begin{hypothesis}[Local coercivity]\label{h:local}
  Assume that there exists \(\lambda_1 > 0\) such that for all
  \(x \in \R^d\) and all \(g=g(v)\) it holds for the case
  \(L = \collFP\) that
  \begin{equation}
    \label{eq:local-coercivity-fp}
    \int_{v \in \R^d}
    \left|
      g - \lAvg{g} \frac{\fSt}{\lAvg{\fSt}}
    \right|^2
    \frac{\dd v}{\fSt}
    \le \lambda_1
    \int_{v \in \R^d}
    \left|
      \nabla_v \left(\frac{g}{\fSt}\right)
    \right|^2
    \fSt\, \dd v
  \end{equation}
  and for the case \(L = \collBGK\) that
  \begin{equation}
    \label{eq:local-coercivity-bgk}
    \begin{aligned}
      &\int_{v \in \R^d}
        \left|
        g - \lAvg{g} \frac{\fSt}{\lAvg{\fSt}}
        \right|^2
        \frac{\dd v}{\fSt} \\
      &\le \frac{\lambda_1}{2}
        \int_{v,v_* \in \R^{d}}
        \left(\frac{g(v)}{\fSt(x,v)}
        - \frac{g(v_*)}{\fSt(x,v_*)}
        \right)^2
        \fSt(x,v) M(x,v_*)\,
        \dd v\, \dd v_*.
    \end{aligned}
  \end{equation}
\end{hypothesis}

\begin{hypothesis}[Spatial coercivity]\label{h:spatial}
  Define the weight \(w=w(x)\) by
  \begin{equation}
    \label{eq:def-weight}
    w(x) = 1 +
    \frac{\int_{|v| \le 1} \fSt |G(x,v)|^2\, \dd v}{\lAvg{\fSt}}
    +
    \left(
      \int_{|v| \le 1}
      \left|\nabla_x\left(\frac{\fSt}{\lAvg{\fSt}}\right)\right|
      \dd v
    \right)^2
    +
    \left|
      \frac{\nabla_x \lAvg{\fSt}}{\lAvg{\fSt}}
    \right|^2
    .
  \end{equation}
  Then assume that there exist a time \(T>0\) and a constant
  \(\lambda_2\) such that for every \(g=g(t,x)\), \(t \in [0,T]\),
  \(x \in \R^d\), there exists a vector field
  \(\vec F : [0,T] \times \R^d \to \R^{1+d}\) satisfying
  \begin{equation}
    \label{eq:solution-bogovoskii}
    \left\{
      \begin{aligned}
        &\nabla_{t,x} \cdot \vec F(t,x) = g(t,x)
        & &\text{in } (t,x) \in [0,T]\times \R^d,\\
        &\vec F(0,x) = \vec F(T,x) = 0 & &\text{on } x \in \R^d
      \end{aligned}
    \right.
  \end{equation}
  and
  \begin{equation}
    \label{eq:bound-bogovoskii}
    \int_{t,x \in [0,T] \times \R^d} \left[|\vec F|^2 w + |\nabla \vec F|^2\right]
    \frac{\dd t\, \dd x}{\lAvg{\fSt}}
    \le \lambda_2
    \int_{t,x \in [0,T] \times \R^d} |g|^2
    \frac{\dd t\, \dd x}{\lAvg{\fSt}}.
  \end{equation}
\end{hypothesis}
\begin{remark}\label{rem:spatial-poincare}
  Under mild regularity assumptions on \(\lAvg{\fSt}\), the
  construction in
  \cite[Section~2]{dietert-hérau-hutridurga-mouhot-2022-preprint-quantitative-geometric-control-linear-kinetic-theory}
  shows that \hypref{h:spatial} is implied by a suitable Poincaré
  inequality, see \cref{sec:bogovskii-poincare}.  In the case that
  \(w \eqsim 1 + |\nabla \log\lAvg{\fSt}|^2\), the required Poincaré
  inequality becomes
  \begin{equation}\label{eq:spatial-poincare}
    \int |h|^2\, w\, \lAvg{\fSt}\, \dd t\, \dd x
    \lesssim
    \int |\nabla h|^2\, \lAvg{\fSt}\, \dd t\, \dd x
  \end{equation}
  for all functions \(h\) with \(\int h \lAvg \fSt = 0\).
\end{remark}

Under these assumption, we show exponential decay towards the
equilibrium distribution. As the overall mass is conserved, this is
equivalent to saying that a perturbation with zero averages converges
to zero.
\begin{theorem}\label{thm:exponential-decay}
  Assume \hypref{h:stationary}, \hypref{h:local}, \hypref{h:spatial}.
  Then there exist constructive constants \(\lambda,C>0\) such that a
  solution \(f=f(t,x,v)\) of \eqref{eq:kinetic-evolution} with initial
  data \(f_\init=f_\init(x,v)\) of zero average
  \(\int_{x,v} f_\init\, \dd x\, \dd v = 0\) satisfies for all \(t>0\)
  \begin{equation}
    \label{eq:exponential-decay}
    \| f \| \le C \ee^{-\lambda t} \|f_\init\|.
  \end{equation}
\end{theorem}

\begin{remark}[Gibbs states]
  The result also covers the traditional Gibbs states, where \(M\) is
  a Gaussian with constant temperature, say
  \(M(x,v) = (2\pi)^{-d/2} \ee^{-v^2/2}\) and \(G = - \nabla_x \phi\)
  for a potential \(\phi\).  Then \(\fSt = Z^{-1} \ee^{-v^2-\phi}\) is
  a stationary state where \(Z\) is a normalisation constant.  In this
  case, the required Poincaré inequality is
  \begin{equation}\label{eq:gibbs-poincare}
    \int |h|^2 (1+|\nabla \phi|^2) \ee^{-\phi}\, \dd t\, \dd x
    \lesssim
    \int |\nabla h|^2 \ee^{-\phi}\, \dd t\, \dd x,
  \end{equation}
  which is classical in this setting and the gained weight
  \(|\nabla \phi|^2\) is well-known.

  In this setting the transport and collision operator vanish
  separately which
  \textcite{bouin-dolbeault-ziviani-2023-preprint-l-hypocoercivity-fokker-planck-gibbs}
  calls factorised Gibbs state.  As in this work, they show decay
  depending on a local coercivity and a spatial dispersion
  corresponding to \eqref{eq:gibbs-poincare}.
\end{remark}

\begin{remark}
  The strategy can incorporate boundary conditions and more general
  collision operators, see
  \cite{dietert-hérau-hutridurga-mouhot-2022-preprint-quantitative-geometric-control-linear-kinetic-theory}
  for such results. By using a transfer along the transport, also
  spatial degenerate settings can be controlled.
\end{remark}

In the case of a weak confinement, the required local coercivity
\hypref{h:local} or the spatial coercivity \hypref{h:spatial} fail for
large velocities and large positions, respectively. Using a uniform
moment bound, the weakened coercivity can be interpolated to obtain
polynomial decay, see
\cite{bouin-dolbeault-ziviani-2023-preprint-l-hypocoercivity-fokker-planck-gibbs}
for a review for Gibbs states.

Here we will show that the method can cover these cases as well. We
will demonstrate this by considering the case of a missing spatial
confinement, where we loose a weight in the spatial coercivity and
assume a uniform moment bound.

\sethypothesistag{\ref*{h:spatial}W}
\begin{hypothesis}[Weak spatial coercivity]\label{h:spatial-weak}
  With the weight \(w\) from \eqref{eq:def-weight} assume that there
  exist a time \(T>0\), constants \(k>\ell>0\) and \(\lambda_2,C_k>0\)
  such that for every \(g=g(t,x)\), \(t \in [0,T]\), \(x \in \R^d\)
  there exists a vector field
  \(\vec F : [0,T] \times \R^d \to \R^{1+d}\) satisfying
  \begin{equation}
    \label{eq:solution-bogovoskii-weak}
    \left\{
      \begin{aligned}
        &\nabla_{t,x} \cdot \vec F(t,x) = g(t,x)
        & &\text{in } (t,x) \in [0,T]\times \R^d,\\
        &\vec F(0,x) = \vec F(T,x) = 0 & &\text{on } x \in \R^d
      \end{aligned}
    \right.
  \end{equation}
  and
  \begin{equation}
    \label{eq:bound-bogovoskii-weak}
    \int_{t,x \in [0,T] \times \R^d} \left[|\vec F|^2 w + |\nabla \vec F|^2\right]
    \frac{\dd t\, \dd x}{\lAvg{\fSt}}
    \le \lambda_2
    \int_{t,x \in [0,T] \times \R^d} |g|^2
    (1+|x|^2)^{\ell}
    \frac{\dd t\, \dd x}{\lAvg{\fSt}}.
  \end{equation}
  Moreover, assume that we have a uniform moment bound of the spatial
  density of the solution as
  \begin{equation}
    \label{eq:spatial-moment-bound}
    \sup_{t\ge 0} \int_{x\in\R^d} \lAvg{f}^2 (1+|x|^2)^k\, \dd x
    \le C_k.
  \end{equation}
\end{hypothesis}

The relation of such a weighted Bogosvkiǐ estimate to the classical
weighted Poincaŕe estimates is discussed in
\cref{sec:bogovskii-poincare}.

\begin{theorem}[Polynomial decay]\label{thm:polynomial-decay}
  Assume \hypref{h:stationary}, \hypref{h:local}, \hypref{h:spatial-weak}.
  Then there exist constructive constants \(\lambda,C>0\) such that a
  solution \(f=f(t,x,v)\) of \eqref{eq:kinetic-evolution} with initial
  data \(f_\init=f_\init(x,v)\) of zero average
  \(\int_{x,v} f_\init\, \dd x\, \dd v = 0\) satisfies for all \(t>0\)
  \begin{equation}
    \label{eq:exponential-decay}
    \| f \| \le \frac{C}{(1+t)^{\frac{k}{2\ell}}} \|f_\init\|.
  \end{equation}
\end{theorem}

\begin{remark}
  A similar statement can be obtained for a weaker coercivity of the
  velocity variable \(v\).
\end{remark}

\subsection{Literature}

The literature on hypocoercivity is enormous by now so that we cannot
give a complete overview. An slightly old overview is given in the
book by \textcite{villani-2009-hypocoercivity} and a newer
introduction can, e.g., be found in the thesis by
\textcite{evans-2019-thesis}.  An overview of the convergence results
for Gibbs states is in
\cite{bouin-dolbeault-ziviani-2023-preprint-l-hypocoercivity-fokker-planck-gibbs}.

The use of the Bogovskiǐ operator is classical in the study of fluid
dynamics, see e.g. the textbook
\cite{galdi-2011-navier-stokes}. However, there it is studied without
a weight which was introduced in
\cite{dietert-hérau-hutridurga-mouhot-2022-preprint-quantitative-geometric-control-linear-kinetic-theory}.
An overview of the used Poincaré inequalities in kinetic theory can,
e.g., be found in the introduction of
\cite{bouin-dolbeault-ziviani-2023-preprint-l-hypocoercivity-fokker-planck-gibbs}.

The field of hypocoercivity is still very active and, in particular,
recently there has been several studies for general states using the
Harris theorem, e.g.,
\cite{canizo-cao-evans-yoldas-2020-hypocoercivity-harriss,cao-2021-fokker-planck,bernou-2023-preprint-asymptotic-behavior-degenerate-linear-kinetic,yoldas-2023-harris}
and using the energy estimates of
\textcite{dolbeault-mouhot-schmeiser-2015-hypocoercivity} in, e.g.,
\cite{bouin-hoffmann-mouhot-2017-exponential,evans-menegaki-2023-preprint-properties-non-equilibrium-steady-states-bgk}.

\section{Proof of exponential decay}
\label{sec:exponential-decay}

In this section, we will prove our main result
\cref{thm:exponential-decay}.  The evolution of
\eqref{eq:kinetic-evolution} is given by a semigroup with generator
\(A\) defined as
\begin{equation}
  \label{eq:def-generator}
  Af = - v \cdot \nabla_x f - \nabla_v \cdot (G f) + Lf.
\end{equation}
The existence of such a semigroup is standard and the semigroup has as
a core smooth functions. Therefore, it suffices to show the a priori
estimates for the claimed result.

Over the function space \(\fsWL^2\) the adjoint operator \(A^*\) is
given by
\begin{equation}
  \label{eq:adjoint-generator}
  A^* = v \cdot \nabla_x - \fSt^{-1} v \cdot \nabla_x \fSt
  + G \cdot \nabla_v
  - \fSt^{-1} G \cdot \nabla_v \fSt
  + L^*,
\end{equation}
where the adjoint for the \(\collFP\) is
\begin{equation}
  \label{eq:adjoint-coll-fp}
  \collFP^* f = \frac{\fSt}{M} \nabla_v \cdot
  \left(M \nabla_v \left(\frac{f}{\fSt}\right)\right)
\end{equation}
and for \(\collBGK\)
\begin{equation}
  \label{eq:adjoint-coll-bgk}
  \collBGK^* f = \fSt \lAvg{f \frac{M}{\fSt}} - f.
\end{equation}
We can then split \(A\) into the symmetric part \(A_s = (A+A^*)/2\)
and the antisymmetric part \(A_a = (A-A^*)/2\). Using that \(\fSt\) is
a stationary solution, i.e.\ \eqref{eq:def-stationary-state} from
\hypref{h:stationary}, yields
\begin{equation}
  \label{eq:generator-sym}
  A_s = - \frac 12 \fSt^{-1} L \fSt + \frac 12 (L+L^*).
\end{equation}

\subsection{Dissipation}

For the evolution, we define the dissipation \(\diss\) as
\begin{equation*}
  \diss := - \frac 12 \frac{\dd}{\dd t} \| f \|_{\fsWL^2}^2
  = -\ip{f}{Af}.
\end{equation*}
The antisymmetric part \(A_a\) vanishes in the dissipation. From the
symmetric part, we find for \(L=\collFP\) that
\begin{equation}
  \label{eq:dissipation-fp}
  \begin{aligned}
    \diss
    &= \frac 12 \int_{x,v \in \R^{d}}
      \left[
      \frac{f^2}{\fSt^2} \collFP \fSt
      - \frac{f}{\fSt} \collFP f
      - \frac{f}{\fSt} \collFP^*f
      \right]
      \dd x\, \dd v\\
    &= \int_{x,v \in \R^{d}}
      \left[
      - \frac{f}{\fSt}
      \nabla_v\left(\frac{f}{\fSt}\right)
      \cdot M
      \nabla_v\left(\frac{\fSt}{M}\right)
      +
      \nabla_v\left(\frac{f}{\fSt}\right)
      \cdot M
      \nabla_v\left(\frac{f}{M}\right)
      \right]
      \dd x\, \dd v\\
    &= \int_{x,v \in \R^{d}}
      \left|
      \nabla_v\left(\frac{f}{\fSt}\right)
      \right|^2
      \fSt
      \dd x\, \dd v\\
  \end{aligned}
\end{equation}
and for \(L = \collBGK\) that
\begin{equation}
  \label{eq:dissipation-bgk}
  \begin{aligned}
    \diss
    &= \frac 12 \int_{x,v \in \R^{d}}
      \left[
      \frac{f^2}{\fSt^2} \collBGK \fSt
      - \frac{f}{\fSt} \collBGK f
      - \frac{f}{\fSt} \collBGK^*f
      \right]
      \dd x\, \dd v\\
    &= \frac 12\int_{x,v \in \R^{d}}
      \left[
      \frac{f^2}{\fSt}
      - f \frac{M}{\fSt} \lAvg{f}
      - \lAvg{f \frac{M}{\fSt}} f
      f^2
      \frac{M \lAvg{\fSt}}{\fSt^2}
      \right]
      \dd x\, \dd v\\
    &= \frac 12 \int_{x,v,v_* \in \R^{d}}
      \left(\frac{f(x,v)}{\fSt(x,v)}
      - \frac{f(x,v_*)}{\fSt(x,v_*)}
      \right)^2
      \fSt(x,v) M(x,v_*)\,
      \dd x\, \dd v\, \dd v_*.
  \end{aligned}
\end{equation}

\subsection{Decay criterion}

By the computation of the dissipation, we see that the dissipation is
non-negative so that we have a contraction semigroup. For exponential
decay, it then suffices to show that for a fixed time \(T>0\) there
exists a constant \(C\) such that any solution satisfies
\begin{equation}\label{eq:exponential-decay-criterion}
  \int_0^T \| f \|_{\fsWL^2}^2\, \dd t
  \le C \int_0^T \diss\, \dd t.
\end{equation}

Indeed if this is true, we can use that the evolution is contractive
to find
\begin{equation*}
  T \| f(T) \|_{\fsWL^2}^2
  \le C \int_0^T \diss\, \dd t
  = C( \| f_\init \|_{\fsWL^2}^2 - \| f(T) \|_{\fsWL^2}^2 )
\end{equation*}
so that
\begin{equation*}
  \| f(T) \|_{\fsWL^2}^2
  \le \left(1+\frac TC\right)^{-1}
  \| f_\init \|_{\fsWL^2}^2.
\end{equation*}
This shows the contraction over the time \(T\) and thus exponential
decay. Hence it only remains to prove
\eqref{eq:exponential-decay-criterion}.

\subsection{Reduction to spatial density}

In order to prove \eqref{eq:exponential-decay-criterion}, we first use
the local coercivity to reduce the problem to the local density. We
find that
\begin{equation*}
  \begin{aligned}
    \int_0^T \| f \|_{\fsWL^2}^2\, \dd t
    &\le
      \int_0^T \int_{x,v\in\R^d}
      2
      \left[
      \left(\lAvg{f}\frac{\fSt}{\lAvg{\fSt}}\right)^2
      +
      \left(f-\lAvg{f}\frac{\fSt}{\lAvg{\fSt}}\right)^2
      \right]
      \frac{\dd x\, \dd v}{\fSt} \dd t \\
    &\le
      2 \int_0^T \int_{x\in\R^d}
      \lAvg{f}^2 \frac{\dd x\, \dd t}{\lAvg{\fSt}}
      + 2 \lambda_1
      \int_0^T \diss\, \dd t,
  \end{aligned}
\end{equation*}
where we used \hypref{h:local} in the second inequality. Hence it
remains to prove that
\begin{equation}\label{eq:spatial-decay-criterion}
  \int_0^T \int_{x\in\R^d}
  \lAvg{f}^2 \frac{\dd x\, \dd t}{\lAvg{\fSt}}
  \lesssim \int_0^T \diss\, \dd t.
\end{equation}

\subsection{Control of spatial density}

In order to control the spatial density \(\lAvg{f}\), we use that the
dissipation gives a control on the gradient \(\nabla_{t,x}\) of
\(\lAvg{f}\). Denoting the time with index \(0\), we will show that
there exist fields \((K_i)_{i=0,\dots,d}\) and \((J_{ij})_{i,j=0,\dots,d}\)
such that for \(i=0,\dots,d\)
\begin{equation}\label{eq:derivatives-density}
  \partial_i \left(\frac{\lAvg{f}}{\lAvg{\fSt}}\right)
  = K_i + \sum_{j=0}^d \partial_j J_{ij}
\end{equation}
and the fields are bounded for \(i=0,\dots,d\) as (recall the weight
\(w\) from \eqref{eq:def-weight} from \hypref{h:spatial})
\begin{equation}\label{eq:bound-k}
  \int_0^T \int_{x\in \R^d} |K_i|^2 \frac{\lAvg{\fSt}}{w}\, \dd x\, \dd t
  \lesssim
  \int_0^T \diss \, \dd t
\end{equation}
and for \(i,j=0,\dots,d\)
\begin{equation}\label{eq:bound-j}
  \int_0^T \int_{x\in \R^d} |J_{ij}|^2 \lAvg{\fSt}\, \dd x\, \dd t
  \lesssim
  \int_0^T \diss \, \dd t.
\end{equation}

The proof is based on a simple lemma, where \(B_1\) denotes the unit
ball in \(\R^d\).
\begin{lemma}\label{thm:construction-test-function}
  Suppose \(\fSt\) satisfies \eqref{eq:lower-bound-stationary-state}
  from \hypref{h:stationary} and take the weight \(w\) of
  \eqref{eq:def-weight} from \hypref{h:spatial}. Then there exist
  functions \(\psi_i=\psi_i(x,v)\) for \(i=0,\dots,d\) with
  \(\supp \psi \subset \R^d \times B_1\) satisfying the bound
  \begin{equation*}
    \sup_{x,v}
    \left( |\psi| + |\nabla_v \psi| + |\nabla_v^2 \psi|
      + \frac{|\nabla_x \psi|}{\sqrt{w}}
    \right)
    \le C
  \end{equation*}
  for a constant \(C\) and the relation
  \begin{equation*}
    \int_{v\in\R^d} \frac{\fSt}{\lAvg{\fSt}} \psi_i\, \dd v
    = \delta_{i0}
  \end{equation*}
  and for \(j=1,\dots,d\)
  \begin{equation*}
    \int_{v\in\R^d} \frac{\fSt}{\lAvg{\fSt}} \psi_i\, v_j\, \dd v
    = \delta_{ij}.
  \end{equation*}
\end{lemma}
\begin{proof}
  Let \(\phi\) be a smooth cutoff to \(B_1\) and use the notation
  \(v_k=1\) if \(k=0\). Then take the basis functions
  \(e_k(v) = v_k \chi(v)\) for \(k=0,\dots,d\) and consider for a
  fixed point \(x\) in space the matrix
  \begin{equation*}
    M_{ij}
    = \int_{v\in\R^d}
    \frac{\fSt}{\lAvg{\fSt}} v_i e_j\, \dd v
    = \int_{v\in\R^d}
    \frac{\fSt}{\lAvg{\fSt}} v_i v_j\, \chi(v)\, \dd v.
  \end{equation*}
  The matrix \(M\) is symmetric and we find the required functions as
  \begin{equation*}
    \psi_i = M_{ij}^{-1} e_j.
  \end{equation*}
  By the assumed lower bound \eqref{eq:lower-bound-stationary-state}
  from \hypref{h:stationary} we have a uniform lower bound on the
  eigenvalues of the matrix \(M\). Hence \(M^{-1}\) is uniformly
  bounded and the bounds on \(\psi\), \(\nabla_v \psi\) and
  \(\nabla_v^2 \psi\) follow. From \eqref{eq:def-weight} from
  \hypref{h:spatial}, we find that \(|\nabla_x M| \lesssim \sqrt{w}\)
  which implies the claimed bound for \(\nabla_x \psi\).
\end{proof}

In order to obtain \eqref{eq:derivatives-density} for \(i=0,\dots,d\)
we note by linearity that
\begin{equation}
  \label{eq:starting-density-control}
  \begin{aligned}
    &\int_{v\in\R^d}
      (\partial_t - A_a)
      \left(
      \lAvg{f} \frac{\fSt}{\lAvg{\fSt}}
      \right)\,
      \frac{\psi_i}{\lAvg{\fSt}}\, \dd v \\
    &=
      \int_{v\in\R^d}
      (\partial_t - A_a)
      \left(
      \lAvg{f} \frac{\fSt}{\lAvg{\fSt}}
      - f
      \right)\,
      \frac{\psi_i}{\lAvg{\fSt}}\, \dd v
      +
      \int_{v\in\R^d}
      (\partial_t - A_a)
      f\,
      \frac{\psi_i}{\lAvg{\fSt}}\, \dd v.
  \end{aligned}
\end{equation}

In order to compute the LHS of \eqref{eq:starting-density-control},
first note that
\begin{equation*}
  A_a f
  = - v \cdot \nabla_x f
  - \nabla_v\cdot(Gf)
  + \frac 12 \fSt^{-1}
  \left(
    v \cdot \nabla_x \fSt + \nabla_v \cdot(G\fSt)
  \right) f
  + \frac 12 (Lf - L^* f).
\end{equation*}
Moreover, as the evolution preserves the overall mass
\(L^* \fSt = 0\).  As \(L\) is only acting in the \(v\) variable, we
find
\begin{equation*}
  (L-L^*)\left(
    \lAvg{f} \frac{\fSt}{\lAvg{\fSt}}
  \right)
  = \frac{\lAvg{f}}{\lAvg{\fSt}}
  (L-L^*) \fSt
\end{equation*}
so that
\begin{equation*}
  \begin{aligned}
    &\int_{v\in\R^d}
      (\partial_t - A_a)
      \left(
      \lAvg{f} \frac{\fSt}{\lAvg{\fSt}}
      \right)\,
      \frac{\psi_i}{\lAvg{\fSt}}\, \dd v \\
    &=
      \int_{v\in\R^d} (\partial_t + v \cdot \nabla_x)
      \left(
      \lAvg{f} \frac{\fSt}{\lAvg{\fSt}}
      \right)\,
      \frac{\psi_i}{\lAvg{\fSt}}\, \dd v
      +
      \int_{v\in\R^d}
      \nabla_v \left(
      G \lAvg{f} \frac{\fSt}{\lAvg{\fSt}}
      \right)
      \frac{\psi_i}{\lAvg{\fSt}}\, \dd v\\
    &\quad -
      \int_{v\in\R^d}
      \frac{\lAvg{f}}{\lAvg{\fSt}}
      \left(
      v \cdot \nabla_x \fSt + \nabla_v \cdot(G\fSt)
      \right)
      \frac{\psi_i}{\lAvg{\fSt}}\, \dd v\\
    &=
      \int_{v\in\R^d} (\partial_t + v \cdot \nabla_x)
      \left(
      \frac{\lAvg{f}}{\lAvg{\fSt}}
      \right)\,
      \fSt
      \frac{\psi_i}{\lAvg{\fSt}}\, \dd v\\
    &= \partial_i
      \left(
      \frac{\lAvg{f}}{\lAvg{\fSt}}
      \right),
  \end{aligned}
\end{equation*}
where we used the property of \(\psi_i\) from
\cref{thm:construction-test-function} in the last equality.

For the first term of the RHS of \eqref{eq:starting-density-control},
we find
\begin{equation*}
  \begin{aligned}
    &\int_{v\in\R^d}
      (\partial_t - A_a)
      \left(
      \lAvg{f} \frac{\fSt}{\lAvg{\fSt}}
      - f
      \right)\,
      \frac{\psi_i}{\lAvg{\fSt}}\, \dd v \\
    &=
      \int (\partial_t + v\cdot\nabla_x)
      \left(
      \lAvg{f} \frac{\fSt}{\lAvg{\fSt}}
      - f
      \right)\,
      \frac{\psi_i}{\lAvg{\fSt}}\, \dd v
      +
      \int \nabla_v\cdot\left(G
      \left(
      \lAvg{f} \frac{\fSt}{\lAvg{\fSt}}
      - f
      \right)\right)\,
      \frac{\psi_i}{\lAvg{\fSt}}\, \dd v
    \\
    &\quad - \frac 12
      \int
      \fSt^{-1}
      \left(
      v \cdot \nabla_x \fSt + \nabla_v \cdot(G\fSt)
      \right)
      \left(
      \lAvg{f} \frac{\fSt}{\lAvg{\fSt}}
      - f
      \right)\,
      \frac{\psi_i}{\lAvg{\fSt}}\, \dd v\\
    &\quad - \frac 12
      \int
      (L-L^*)
      \left(
      \lAvg{f} \frac{\fSt}{\lAvg{\fSt}}
      - f
      \right)\,
      \frac{\psi_i}{\lAvg{\fSt}}\, \dd v\\
    &= \partial_t
      \int
      \left(
      \lAvg{f} \frac{\fSt}{\lAvg{\fSt}}
      - f
      \right)\,
      \frac{\psi_i}{\lAvg{\fSt}}\, \dd v
      + \nabla_x \cdot
      \int
      \left(
      \lAvg{f} \frac{\fSt}{\lAvg{\fSt}}
      - f
      \right)\,
      \frac{v \psi_i}{\lAvg{\fSt}}\, \dd v\\
    &\quad -
      \int
      \left(
      \lAvg{f} \frac{\fSt}{\lAvg{\fSt}}
      - f
      \right)\,
      v \cdot \nabla_x\left(\frac{\psi_i}{\lAvg{\fSt}}\right)\, \dd v
      -
      \int
      \left(
      \lAvg{f} \frac{\fSt}{\lAvg{\fSt}}
      - f
      \right)\,
      \frac{G \cdot \nabla_v \psi_i}{\lAvg{\fSt}}\, \dd v\\
    &\quad- \frac 12
      \int
      (\fSt^{-1} L \fSt + L - L^*)
      \left(
      \lAvg{f} \frac{\fSt}{\lAvg{\fSt}}
      - f
      \right)\,
      \frac{\psi_i}{\lAvg{\fSt}}\, \dd v.
  \end{aligned}
\end{equation*}
For the second term on the RHS of \eqref{eq:starting-density-control},
we use that \((\partial_t - A_a)f = A_s f\) so that
\begin{equation*}
  \begin{aligned}
    \int_{v\in\R^d}
    (\partial_t - A_a)
    f\,
    \frac{\psi_i}{\lAvg{f}}\, \dd v
    &=
      \frac 12
      \int
      (- \fSt^{-1} L \fSt + L + L^*) f\,
      \frac{\psi_i}{\lAvg{\fSt}}.
  \end{aligned}
\end{equation*}
Hence we find a representation \eqref{eq:derivatives-density} with
(using again the notation \(v_j=1\) for \(j=0\))
\begin{equation}\label{eq:def-j}
  J_{ij}
  =
  \int
  \left(
    \lAvg{f} \frac{\fSt}{\lAvg{\fSt}}
    - f
  \right)\,
  \frac{v_j \psi_i}{\lAvg{\fSt}}\, \dd v
\end{equation}
and
\begin{equation}\label{eq:def-k}
  \begin{aligned}
    K_i
    &=
      -
      \int
      \left(
      \lAvg{f} \frac{\fSt}{\lAvg{\fSt}}
      - f
      \right)\,
      v \cdot \nabla_x\left(\frac{\psi_i}{\lAvg{\fSt}}\right)\, \dd v\\
    &\quad-
      \int
      \left(
      \lAvg{f} \frac{\fSt}{\lAvg{\fSt}}
      - f
      \right)\,
      \frac{G \cdot \nabla_v \psi_i}{\lAvg{\fSt}}\, \dd v\\
    &\quad- \frac 12
      \int
      (\fSt^{-1} L \fSt + L - L^*)
      \left(
      \lAvg{f} \frac{\fSt}{\lAvg{\fSt}}
      - f
      \right)\,
      \frac{\psi_i}{\lAvg{\fSt}}\, \dd v \\
    &\quad+
      \frac 12
      \int
      (- \fSt^{-1} L \fSt + L + L^*) f\,
      \frac{\psi_i}{\lAvg{\fSt}}\, \dd v.
  \end{aligned}
\end{equation}

We can bound \(J_{ij}\) as
\begin{equation*}
  |J_{ij}|^2 \lAvg{\fSt}
  \le
  \left(
    \int \left(
      \lAvg{f} \frac{\fSt}{\lAvg{\fSt}}
      - f
    \right)^2 \, \frac{\dd v}{\fSt}
  \right)
  \left(
    \int
    \frac{|v_j \psi_i|^2}{\lAvg{\fSt}^2}
    \fSt\, \dd v\, \lAvg{\fSt}
  \right).
\end{equation*}
The second bracket is uniformly bounded so that the local coercivity
\hypref{h:local} implies the bound \eqref{eq:bound-j}.

For the first term of \(K_i\) we find the bound
\begin{equation*}
  \begin{aligned}
    &\left|
      \int
      \left(
      \lAvg{f} \frac{\fSt}{\lAvg{\fSt}}
      - f
      \right)\,
      v \cdot \nabla_x\left(\frac{\psi_i}{\lAvg{\fSt}}\right)\, \dd v
      \right|^2 \frac{\lAvg{\fSt}}{w}\\
    &\le
      \left(
      \int \left(
      \lAvg{f} \frac{\fSt}{\lAvg{\fSt}}
      - f
      \right)^2 \, \frac{\dd v}{\fSt}
      \right)
      \left(
      \int
      |v|^2
      \left|\nabla_x \left(\frac{\psi_i}{\lAvg{\fSt}}\right) \right|^2
      \fSt\, \dd v\, \frac{\lAvg{\fSt}}{w}
      \right).
  \end{aligned}
\end{equation*}
By the control of \(\nabla_x \psi\) from
\cref{thm:construction-test-function} and the definition of \(w\), the
local coercivity again yields the claimed bound of \eqref{eq:bound-k}
for this term.

The second term of \(K_i\) is bounded similarly using the definition
of \(w\).

In the case \(L = \collFP\), we find for the third term that
\begin{equation*}
  \begin{aligned}
    &\int
      (\fSt^{-1} L \fSt + L - L^*)
      \left(
      \lAvg{f} \frac{\fSt}{\lAvg{\fSt}}
      - f
      \right)\,
      \frac{\psi_i}{\lAvg{\fSt}}\, \dd v \\
    &= 2
      \int
      \nabla_v\left(\fSt^{-1}
      \left(
      \lAvg{f} \frac{\fSt}{\lAvg{\fSt}}
      - f
      \right)
      \right)
      \fSt \frac{\nabla_v \psi}{\lAvg{\fSt}} \\
    &\quad+ 2
      \int
      \fSt^{-1}
      \left(
      \lAvg{f} \frac{\fSt}{\lAvg{\fSt}}
      - f
      \right)
      \frac{\fSt}{M}
      \frac{\nabla_v(M\nabla \psi)}{\lAvg{\fSt}}.
  \end{aligned}
\end{equation*}
For the first term note that
\( \nabla_v\left(\fSt^{-1} \left( \lAvg{f} \frac{\fSt}{\lAvg{\fSt}} -
    f \right) \right) = \nabla_v(f/\fSt)\) which is controlled by the
dissipation in this case. For the second term use
\eqref{eq:regularity-m-fp}. Hence we find the required bound of
\eqref{eq:bound-k} for this term.

In the case \(L = \collBGK\), we find for the third term that
\begin{equation*}
  \begin{aligned}
    &\int
      (\fSt^{-1} L \fSt + L - L^*)
      \left(
      \lAvg{f} \frac{\fSt}{\lAvg{\fSt}}
      - f
      \right)\,
      \frac{\psi_i}{\lAvg{\fSt}}\, \dd v \\
    &= \int \fSt^{-1} (\lAvg{\fSt}M-\fSt)
      \left(
      \lAvg{f} \frac{\fSt}{\lAvg{\fSt}}
      - f
      \right)\,
      \frac{\psi_i}{\lAvg{\fSt}}\, \dd v \\
    &\quad -
      \int
      \fSt
      \lAvg{
      \frac{M}{\fSt}
      \left(
      \lAvg{f} \frac{\fSt}{\lAvg{\fSt}}
      - f
      \right)
      }
      \frac{\psi_i}{\lAvg{\fSt}}\, \dd v.
  \end{aligned}
\end{equation*}
The bound for the first term follows form \hypref{h:local} and
\eqref{eq:regularity-m-bgk}.  The last term can be rewritten as
\begin{equation*}
  \begin{aligned}
    &\int
      \fSt
      \lAvg{
      \frac{M}{\fSt}
      \left(
      \lAvg{f} \frac{\fSt}{\lAvg{\fSt}}
      - f
      \right)
      }
      \frac{\psi_i}{\lAvg{\fSt}}\, \dd v \\
    &=
      \left(
      \int_{v,v_*}
      \frac{M(v)\fSt(x,v_*)}{\lAvg{\fSt}}
      \left[
      \frac{f(x,v_*)}{\fSt(x,v_*)}
      -
      \frac{f(x,v)}{\fSt(x,v)}
      \right]
      \right)
      \left(
      \int \frac{\fSt}{\lAvg{\fSt}} \psi_i \dd v
      \right)
  \end{aligned}
\end{equation*}
so that we find the claimed bound with the dissipation.

For the fourth term, note in the case \(L=\collFP\) that
\begin{equation*}
  \begin{aligned}
    \int
    (- \fSt^{-1} L \fSt + L + L^*) f\,
    \frac{\psi}{\lAvg{\fSt}}\, \dd v
    = -2
    \int
    \nabla\left(\frac{f}{\fSt}\right)
    \fSt \frac{\nabla \psi}{\lAvg{\fSt}}\, \dd v
  \end{aligned}
\end{equation*}
which yields the claimed bound. In the case \(L = \collBGK\) we find
\begin{equation*}
  \begin{aligned}
    &\int
    (- \fSt^{-1} L \fSt + L + L^*) f\,
    \frac{\psi}{\lAvg{\fSt}}\, \dd v\\
    &=
    \int_{v,v_*}
    \left[
    \frac{f(x,v_*)}{\fSt(x,v_*)}
    -
    \frac{f(x,v)}{\fSt(x,v)}
    \right]
    [M(x,v) \fSt(x,v_*)
    + M(x,v_*) \fSt(x,v)]
    \frac{\psi_i(v)}{\lAvg{\fSt}}
    \,\dd v\, \dd v_*
  \end{aligned}
\end{equation*}
which again yields the claimed bound by the dissipation.

\subsection{Conclusion}

As we assume zero overall mass, we find that
\begin{equation*}
  \int_{[0,T]\times\R^d} \lAvg{f}\, \dd t\, \dd x = 0
\end{equation*}
so that we can apply \hypref{h:spatial} to find \(\vec F\) such that
\begin{equation*}
  \nabla_{t,x} \cdot \vec F = \lAvg{f}
\end{equation*}
and
\begin{equation}\label{eq:bounds-applied-F}
  \int_{t,x \in [0,T] \times \R^d} \left[|\vec F|^2 w + |\nabla \vec F|^2\right]
  \frac{\dd t\, \dd x}{\lAvg{\fSt}}
  \le \lambda_2
  \| \lAvg{f} \|_{\fsWL^2([0,T]\times\R^d)}^2,
\end{equation}
where we use the analogous weighted norm over \(t\) and \(x\) as
\begin{equation*}
  \| \lAvg{f} \|_{\fsWL^2([0,T]\times\R^d)}^2
  = \int_{t,x \in [0,T] \times \R^d}
  \lAvg{f}^2
  \frac{\dd t\, \dd x}{\lAvg{\fSt}}.
\end{equation*}

Using \eqref{eq:derivatives-density} we therefore find
\begin{equation*}
  \begin{aligned}
    \| \lAvg{f} \|_{\fsWL^2([0,T]\times\R^d)}^2
    &=
    \int_{[0,T]\times\R^d} \frac{\lAvg{f}}{\lAvg{\fSt}} \nabla_{t,x}
      \cdot \vec F\, \dd t\, \dd x \\
    &= - \int_{[0,T]\times\R^d}
      (K_i + \partial_j J_{ij}) F_i \, \dd t\, \dd x \\
    &= - \int_{[0,T]\times\R^d} K_i\, F_i \, \dd t\, \dd x
      + \int_{[0,T]\times\R^d} J_{ij} \partial_j F_i \, \dd t\, \dd x\\
    &\le \left(
      \int |K|^2 \frac{\lAvg{\fSt}}{w}
      \right)^{1/2}
      \left(
      \int |\vec F|^2 \frac{w}{\lAvg{\fSt}}
      \right)^{1/2}\\
    &\quad +
      \left(
      \int |J|^2 \lAvg{\fSt}
      \right)^{1/2}
      \left(
      \int |\nabla \vec F|^2 \frac{1}{\lAvg{\fSt}}
      \right)^{1/2}.
  \end{aligned}
\end{equation*}
Using the bound \eqref{eq:bounds-applied-F} for \(\vec F\) and
\eqref{eq:bound-k} for \(K\) and \eqref{eq:bound-j} for \(J\) we
therefore find
\begin{equation*}
  \| \lAvg{f} \|_{\fsWL^2}^2
  \lesssim \int_0^T \diss\, \dd t.
\end{equation*}
This was the remaining bound to be proven
\eqref{eq:spatial-decay-criterion} for the exponential decay. Hence
the proof of exponential decay (\cref{thm:exponential-decay}) is
finished.

\section{Weak spatial confinement}

In this section, we prove \cref{thm:polynomial-decay}.  The proof
commences as for the exponential decay in
\cref{sec:exponential-decay}.  It runs through unchanged until the
last step, where, by the loss of weight, we only find the control
\begin{equation*}
  \| \lAvg{f} \|_{\fsWL^2([0,T]\times\R^d)}^2
  \lesssim
  \left(
    \int_0^T \diss\, \dd t
  \right)^{1/2}
  \| \lAvg{f} (1+x^2)^{\ell/2}\|_{\fsWL^2([0,T]\times\R^d)}
\end{equation*}
By interpolating
\(\| \lAvg{f} (1+x^2)^{\ell/2}\|_{\fsWL^2([0,T]\times\R^d)}\) between
\(\| \lAvg{f} \|_{\fsWL^2([0,T]\times\R^d)}\) and
\({\| \lAvg{f} (1+x^2)^{k/2}\|_{\fsWL^2([0,T]\times\R^d)}}\), which we
assume to be bounded in \hypref{h:spatial-weak}, we find
\begin{equation*}
  \| \lAvg{f} \|_{\fsWL^2([0,T]\times\R^d)}^{2(1+a)}
  \lesssim
  \int_0^T \diss\, \dd t,
\end{equation*}
where \(a=\ell/k\).

Combining the estimates, we therefore find after the time \(T\) that
\begin{equation*}
  \epsilon
  \min\left(
    \frac{\| f \|_{\fsWL^2([0,T]\times\R^{2d})}^{2}}{T},
    \left(\frac{\| f \|_{\fsWL^2([0,T]\times\R^{2d})}^{2}}{T}\right)^{1+a}
  \right)
  \le
  \int_0^T \diss\, \dd t,
\end{equation*}
where we may assume that \(\epsilon < 1\). Then we find exponential
decay as before as long as \(\| f \|_{\fsWL^2} \ge 1\) which is
faster than the claimed algebraic decay.

Hence it remains to show the algebraic decay when
\(\| f \|_{\fsWL^2} \le 1\). For this, let us denote the values at
times \(t=nT\), \(n \in \N\), as
\begin{equation*}
  Y_n = \| f(t=nT) \|_{\fsWL^2}
\end{equation*}
where we restrict to \(Y_n \le 1\). Then the above estimate shows that
\begin{equation*}
  \epsilon Y_{n+1}^{2(1+a)}
  \le Y_n^2 - Y_{n+1}^2.
\end{equation*}
This inequality implies elementary that
\begin{equation}\label{eq:decay-y-n}
  Y_{n+1}^2 \le Y_n^2 - \epsilon 2^{-2(1+a)} Y_n^{2(1+a)}.
\end{equation}
To see this, one can assume that we find a bound with
\(Y_{n+1} \ge Y_n/2\) and then replace \(Y_{n+1}^{2(1+a)}\) by
\(2^{-2(1+a)} Y_n\).  Then the assumption can be verified for the
assumed bound.

Then \eqref{eq:decay-y-n} implies that
\begin{equation*}
  \frac{1}{Y_{n+1}^{2a}}
  \ge \frac{1}{Y_n^{2a}} \;
  \left(
    1 - \epsilon 2^{-2(1+a)} Y_n^{2a}
  \right)^{-2a}
\end{equation*}
as
\(\left( 1 - \epsilon 2^{-2(1+a)} Y_n^{2a} \right)^{-2a} \ge 1 + 2a
\epsilon 2^{-2(1+a)} Y_n^{2a}\)
so that we have found
\begin{equation*}
  \frac{1}{Y_{n+1}^{2a}}
  \ge \frac{1}{Y_n^{2a}}
  +
  2a
  \epsilon 2^{-2(1+a)}.
\end{equation*}
This implies that
\begin{equation*}
  Y_n \lesssim (1+n)^{-\frac{1}{2a}},
\end{equation*}
which is the claimed algebraic decay.

\appendix

\section{Relation of Bogovskii inequality and Poincaré inequality}
\label{sec:bogovskii-poincare}

In this appendix, we discuss the relation of the existence of a
suitable Poincaré inequality and the spatial assumption
\hypref{h:spatial} and \hypref{h:spatial-weak}, respectively. For
details and further discussion, we refer to
\cite{dietert-hérau-hutridurga-mouhot-2022-preprint-quantitative-geometric-control-linear-kinetic-theory}.

We collect the basic construction in the following lemma, which
incorporates possible weighted weaker forms.
\begin{lemma}\label{thm:poincare-to-bogovskii}
  Suppose a domain \(\sU \subset \R^d\), \(d \in \N\), with a nice
  boundary and a potential \(\Phi \in C^1(\R^d)\) satisfying
  \(|\nabla^2 \Phi| \lesssim 1 + |\nabla \Phi|\), a weight \(W \ge 0\)
  and constant \(C_P>0\) such that for all \(h : \R^d \to \R\) with
  \(\int_{\sU} h\, \ee^{-\Phi} = 0\) it holds that
  \begin{equation*}
    \int_{\sU} |h|^2\, W\, \ee^{-\Phi}
    \le C_P
    \int_{\sU} |\nabla h|^2\, \ee^{-\Phi}.
  \end{equation*}
  Then there exists a constant \(C_B\) so that for every
  \(g : \sU \to \R\) with \(\int_{\sU} g = 0\), there exists a
  vector field \(\vec F : \sU \to \R^d\) such that
  \begin{equation*}
    \left\{
      \begin{aligned}
        &\nabla \cdot \vec F = g
        & &\text{in } \sU\\
        & \vec F = 0
        & &\text{on } \partial \sU
      \end{aligned}
    \right.
  \end{equation*}
  and
  \begin{equation*}
    \int_{\sU} |\vec F|^2\, \ee^{\Phi}
    \le C_B
    \int_{\sU} |g|^2\, \frac{\ee^{\Phi}}{W}
  \end{equation*}
  and
  \begin{equation*}
    \int_{\sU} |\nabla \vec F|^2\, \frac{\ee^{\Phi}}{1+|\nabla \Phi|^2}
    \le C_B
    \int_{\sU} |g|^2\, \left(\frac 1W + \frac{1}{1+|\nabla \Phi|^2}\right) \ee^{\Phi}.
  \end{equation*}
\end{lemma}
For a discussion of possible boundaries, we refer to
\cite{dietert-hérau-hutridurga-mouhot-2022-preprint-quantitative-geometric-control-linear-kinetic-theory}.
In this work, we only apply it to \(\sU = [0,T] \times \R^d\) where
the boundary at \(t=0\) and \(t=T\). For the details, we again refer
to
\cite{dietert-hérau-hutridurga-mouhot-2022-preprint-quantitative-geometric-control-linear-kinetic-theory}
and just sketch the main arguments with the more general weights here.
\begin{proof}[Proof sketch]
  We first solve the elliptic problem
  \begin{equation*}
    \left\{
      \begin{aligned}
        &\nabla \cdot (\ee^{-\Phi} \nabla h) = g
        & &\text{in } \sU\\
        &\vec{n} \cdot (\ee^{-\Phi} \nabla h) = 0
        & &\text{on } \partial\sU\\
      \end{aligned}
    \right.
  \end{equation*}
  for \(h\) with \(\int_{\sU} h\, \ee^{-\Phi}=0\). By the assumed
  Poincaré inequality, we find a solution and defining
  \(\vec F_0 = \ee^{-\Phi} \nabla h\) it yields the bound
  \begin{equation*}
    \int |\vec F_0|^2 \ee^{\Phi}
    \lesssim
    \int |g|^2\, \frac{\ee^{\Phi}}{W}
  \end{equation*}
  Moreover, \(\nabla \cdot \vec F_0 = g\).

  We then find a covering \((B_k)_k\) of the domain \(\sU\) with a
  corresponding partition of unity \((\theta_k)_k\) where each
  component is of diameter comparable to
  \((1+|\nabla \Phi|^2)^{-1/2}\) so that in each component
  \(\ee^{\Phi}\) is like a constant in the sense that the weight
  \(\ee^{\Phi}\) is only varying by a uniformly bounded
  factor. Moreover, we can ensure that
  \(|\nabla \theta_k| \lesssim (1+|\nabla \Phi|^2)^{1/2}\).

  On each component consider
  \(g_k = \nabla \cdot (\theta_k \vec F_0)\) which satisfies
  \(\int g_k = 0\). Hence we can find on each component \(B_k\) a
  vector field \(\vec F_k\) vanishing outside the component \(B_k\) so
  that \(\nabla \cdot \vec F_k = g_k\) and
  \begin{equation*}
    \| \vec F_k \|_{\fsL^2(B_k)} \lesssim \| \theta_k \vec F_0 \|_{\fsL^2(B_k)}
    \quad\text{and}\quad
    \| \nabla \vec F_k \|_{\fsL^2(B_k)} \lesssim \| g_k \|_{\fsL^2(B_k)}.
  \end{equation*}

  Then \(\vec F = \sum_k \vec F_k\) is the sought vector field.
\end{proof}

For verifying \hypref{h:spatial} as in \cref{rem:spatial-poincare}, we
apply the above lemma with \((t,x) \in [0,T] \times \R^d\) and take
\(\Phi\) such that \(\ee^{\Phi}\) is comparable to \(w \lAvg{\fSt}\)
and chose \(W=w\).

In the case of a weaker spatial confinement, we take as before
\(\Phi\) and only assume for \(\ell > 0\) the weaker Poincaré
inequality
\begin{equation*}
  \int_{\sU} |h|^2 \frac{w}{(1+|x|^2)^\ell} \ee^{-\Phi}
  \lesssim
  \int_{\sU} |\nabla h|^2 \ee^{-\Phi}.
\end{equation*}
Then the previous lemma with \(W=w(1+|x|^2)^\ell\) implies
\hypref{h:spatial-weak}.

\printbibliography

\end{document}